\RequirePackage{etex}
\documentclass{amsart}
\usepackage{amsmath, amsthm, amssymb, mathrsfs, eucal, epsfig}
\usepackage{dcpic, pictex}

\begin{document}

\newtheorem{theorem}{Theorem}[section]
\newtheorem{lemma}[theorem]{Lemma}
\newtheorem{proposition}[theorem]{Proposition}
\newtheorem{corollary}[theorem]{Corollary}
\newtheorem{conjecture}[theorem]{Conjecture}
\newtheorem{question}[theorem]{Question}
\newtheorem{problem}[theorem]{Problem}
\newtheorem*{claim}{Claim}
\newtheorem*{criterion}{Criterion}
\newtheorem*{main_rat_thm}{Rationality Theorem}

\theoremstyle{definition}
\newtheorem{definition}[theorem]{Definition}
\newtheorem{construction}[theorem]{Construction}
\newtheorem{notation}[theorem]{Notation}

\theoremstyle{remark}
\newtheorem{remark}[theorem]{Remark}
\newtheorem{example}[theorem]{Example}

\numberwithin{equation}{subsection}

\def\Z{\mathbb Z}
\def\R{\mathbb R}
\def\Q{\mathbb Q}
\def\D{\mathcal D}
\def\E{\mathcal E}
\def\RR{\mathcal R}
\def\P{\mathcal P}
\def\F{\mathcal F}

\def\cl{\textnormal{cl}}
\def\scl{\textnormal{scl}}
\def\homeo{\textnormal{Homeo}}
\def\rot{\textnormal{rot}}

\def\Id{\textnormal{Id}}
\def\SL{\textnormal{SL}}
\def\PSL{\textnormal{PSL}}
\def\length{\textnormal{length}}
\def\fill{\textnormal{fill}}
\def\rank{\textnormal{rank}}
\def\til{\widetilde}

\title{Surface subgroups from homology}
\author{Danny Calegari}
\address{Department of Mathematics \\ Caltech \\
Pasadena CA, 91125}
\email{dannyc@its.caltech.edu}

%\date{3/31/2008, Version 0.04}

\begin{abstract}
Let $G$ be a word-hyperbolic group, obtained as a graph of free groups
amalgamated along cyclic subgroups. If $H_2(G;\Q)$ is nonzero, then
$G$ contains a closed hyperbolic surface subgroup. Moreover, the unit
ball of the Gromov--Thurston norm on $H_2(G;\R)$ is a finite-sided rational
polyhedron.
\end{abstract}

\maketitle

\section{Introduction}

A famous question of Gromov (see \cite{Bestvina_problem})
asks whether every one-ended non-elementary
word-hyperbolic group contains a closed hyperbolic surface subgroup.
Almost nothing is known about this question in general. Gordon--Long--Reid \cite{GLR}
answer the question affirmatively for Coxeter groups and some Artin groups.

Bestvina remarks that Gromov's question is inspired by the well-known virtual
Haken conjecture in $3$-manifold topology. The case of $3$-manifold groups is
instructive. If $M$ is an aspherical $3$-manifold, every integral homology
class in $H_2(M;\Z)$ is represented by an embedded surface $S$. If $S$
is not $\pi_1$-injective, Dehn's lemma (see \cite{Hempel}, Chapter~4)
implies that $S$ can be compressed, reducing $-\chi(S)$.
By the hypothesis that $M$ is aspherical, 
after finitely many compressions, one obtains a $\pi_1$-injective
surface representing the given homology class. 

For more general classes of groups,
no tool remotely resembling Dehn's lemma exists. Nevertheless one can consider
the following strategy. Let $X$ be a $K(G,1)$, and let $A$ be a rational
homology class in $H_2(X;\Q)$. Suppose one can find a map of a closed surface
$f:S \to X$ with no spherical components,
representing $n(S)A$ in $H_2(X)$ for some integer $n(S)$, which
realizes the infimum of $-\chi(S)/n(S)$ over all surfaces and all integers $n$.
Then $f_*:\pi_1(S) \to \pi_1(X) = G$ is injective. For, otherwise, one could
find an essential loop $\alpha \subset S$ in the kernel of $f_*$ and (by Scott \cite{Scott})
find a suitable finite cover $S'$ of $S$ to which $\alpha$ lifts as an embedded loop.
Then $f':S' \to X$ could be compressed along $\alpha$, producing a new surface $S''$
representing $n(S'')A$ in homology, and satisfying
$-\chi(S'')/n(S'') < -\chi(S)/n(S)$, contrary to hypothesis. This infimal quantity
is called the {\em Gromov--Thurston norm} of the homology class $A$ (see \cite{Gromov_bounded}
or \cite{Thurston_norm} for an introduction to Gromov--Thurston norms and
bounded cohomology). In words,
if a map from a surface to $X$ realizes the Gromov--Thurston norm in a given projective
homology class, it is injective.

It is therefore an intriguing question to understand for which groups $G$ and
which homology classes in $H_2(G;\Q)$ one can find maps of surfaces (projectively)
realizing the Gromov--Thurston norm. In this paper we show that if $G$ is a group
obtained as a graph of free groups amalgamated along cyclic subgroups, and $A \in H_2(G;\Q)$
is a homology class with nonzero Gromov--Thurston norm, then some map of a surface
to a $K(G,1)$ realizes the Gromov--Thurston norm in the projective class of
$A$, and therefore $G$ contains a closed hyperbolic surface subgroup. The method
of proof is to localize the problem to finding norm minimizers for
suitable relative homology classes in the free factors. The relative Gromov--Thurston norm
(after normalization) turns out to be equal to the so-called {\em stable commutator}
norm, introduced in \cite{Calegari_scl}, and studied in free groups in
\cite{Calegari_pickle}. A consequence of the main theorem of \cite{Calegari_pickle} is
that extremal surfaces for the stable commutator norm
exist in every rational relative homology class in a free group. These extremal
surfaces can be glued together to produce extremal (closed) surface subgroups in $G$.
A more careful analysis reveals that the Gromov--Thurston norm on $G$ is piecewise
rational linear, and if $G$ is word-hyperbolic, the unit ball is a finite-sided rational
polyhedron.

\section{The $\scl$ norm}

\subsection{Commutator length}

If $G$ is a group and $g \in [G,G]$, the {\em commutator length} of $g$ (denoted
$\cl(g)$) is the smallest number of commutators in $G$ whose product is equal to $G$, and
the {\em stable commutator length} of $g$ (denoted $\scl(g)$) is the limit
$$\scl(g) = \lim_{n \to \infty} \frac {\cl(g^n)} n$$
Geometrically, $\cl(g)$ is the least genus of a surface group which bounds
$g$ homologically. Since genus is not multiplicative under covers but Euler
characteristic is, one can derive a formula for $\scl$ in terms of Euler
characteristic; we give such a formula in Definition~\ref{geometric_scl_definition}
below.

Stable commutator length is, in a sense to be made precise shortly, 
a kind of relative Gromov--Thurston norm. 

The following material is largely drawn from \cite{Calegari_pickle}, \S~2.4. Also
see \cite{Calegari_scl}, \S~2.6 and \cite{Bavard}, \S~3.

\begin{definition}
Let $S$ be a compact orientable surface. Define
$$\chi^-(S) = \sum_{S_i} \min(0,\chi(S_i))$$
where the sum is taken over connected components $S_i$ of $S$.
\end{definition}

\begin{definition}\label{geometric_scl_definition}
Let $G$ be a group. Let $g_1,\cdots,g_m$ be elements in $G$ (not necessarily
distinct). Let $X$ be a connected CW complex with $\pi_1(X) = G$.
Further, for each $i$, let $\gamma_i:S^1 \to X$
be a loop in $X$ in the free homotopy class corresponding to the
conjugacy class of $g_i$. 

If $S$ is an orientable surface, a map $f:S \to X$ is {\em admissible} of
{\em degree $n(S)$} for some positive integer $n(S)$ if there is
a commutative diagram
\[\begindc{0}[4]
	\obj(0,20)[S]{$S$}
	\obj(15,20)[dS]{$\partial S$}
	\obj(30,20)[C]{$\coprod_i S^1$}
	\obj(30,10)[X]{$X$}
	\mor(15,20)(0,20){$i$}[\atright,\solidarrow]
	\mor(15,20)(28,20){$\partial f$}
	\mor(30,20)(30,10){$\coprod_i \gamma_i$}
	\mor(0,20)(30,10){$f$}[\atright,\solidarrow]
\enddc\]
so that the homology class of
$\partial f_*[\partial S]$ is equal to $n(S)$ times the fundamental class of
$\coprod_i S^1$ in $H_1$.

Then define
$$\scl(g_1 + g_2 + \cdots + g_m) = \inf_S \frac {-\chi^-(S)} {2n(S)}$$
where the infimum is taken over all admissible maps of surfaces.
If no admissible surfaces exist, set $\scl(\sum_i g_i) = \infty$.
\end{definition}

\begin{remark}
If $X$ has enough room (e.g. if $X$ is a manifold of dimension $>2$)
then the maps $\gamma_i$ can be taken to be embeddings, and one can
speak of the maps $\gamma_i$ and their images interchangeably. In this
context, one can think of an admissible map as a map of pairs 
$(S,\partial S) \to (X,\cup_i \gamma_i)$ which wraps $\partial S$ around
each $\gamma_i$ with total degree $n(S)$.
\end{remark}

\begin{remark}
When $g \in [G,G]$, the value of $\scl(g)$ is the same with either
definition above.
\end{remark}

The function $\scl$ can be extended to integral group $1$-chains, by the formula
$$\scl(\sum n_ig_i):= \scl(\sum g_i^{n_i})$$
and extended to rational chains by linearity, and to real chains by continuity.
It is finite exactly on group $1$-chains which are boundaries of group $2$-chains;
in other words, $\scl$ defines a pseudo-norm on the real vector space $B_1(G;\R)$,
hereafter denoted $B_1(G)$.

Notice that $\scl$ is, by construction, a homogeneous class function in each
variable separately. If $H$ denotes the subspace of $B_1(G;\R)$ spanned by
elements of the form $g - hgh^{-1}$ and $g^n - ng$ for $g,h \in G$ and $n \in \Z$,
then $\scl$ descends to a pseudo-norm on $B_1(G)/H$.

\subsection{Comparison with Gromov and filling norms}

Let $C_*(G;\R)$ be the bar chain complex of a group (see e.g. \cite{MacLane} Ch.~IV, \S~5 for
details). In the sequel, the coefficient group $\R$ is understood where omitted. 
There is a natural basis for $C_i(G)$ in each dimension, and each $C_i(G)$ becomes
a Banach space with respect to the natural $L^1$ norm. This norm induces
a pseudo-norm on (group) homology, called the {\em Gromov norm} (or {\em $L^1$ norm})
defined by
$$\|[A]\|_1 = \inf_{C \in [A]} \|C\|_1$$
where the infimum ranges over all cycles $C$ representing a homology class $[A]$.

If $X$ is a $K(G,1)$, the norm on $H_2(G;\Q)$ may be calculated geometrically
by the formula
$$\|[A]\|_1 = \inf_S \frac {-2\chi^-(S)} {n(S)}$$
where the infimum is taken over all closed oriented surfaces $S$ mapping
to $X$ by $f:S \to X$ for which $f_*[S] = n(S)[A]$ for some integer $n(S)$, and then
extended to $H_2(G;\R)$ by continuity; see \cite{Gabai_foliations}, Corollary~6.18.
Also compare with Definition~\ref{geometric_scl_definition}.

\vskip 12pt

There is a natural norm on $B_1(G)$, called the (Gersten) filling norm,
introduced in \cite{Gersten}, defined by the formula
$$\|A\|_\partial = \inf_{\partial C = A} \|C\|_1$$
where $\|\cdot\|_1$ denotes the $L^1$ norm on group $2$-chains. Let $\fill(\cdot)$
be the homogenization of $\|\cdot\|_\partial$; i.e.
$$\fill(\sum t_i g_i) = \lim_{n \to \infty} \frac {\|\sum t_i g_i^n\|_\partial} n$$
where $t_i \in \R$ and $g_i \in G$. Then $\fill$ descends to a function on
$B_1(G)/H$ and satisfies
$$\scl(A) = \frac {\fill(A)} 4$$
For $A=g$ for $g \in [G,G]$, this is proved in Bavard \cite{Bavard}; the
general case follows basically the same argument, and is found in \cite{Calegari_scl},
\S~2.6. The factor of $\frac 1 4$ arises because $\fill$ counts triangles, whereas
$\scl$ counts genus. This explains the sense in which $\scl$ can be thought
of as a relative Gromov--Thurston norm.

\subsection{Extremal surfaces}

Given an integral chain $\sum n_i g_i$, an admissible surface is {\em extremal}
if it realizes
$$\scl(\sum n_i g_i) = \frac {-\chi^-(S)} {2n(S)}$$
The {\em Rationality Theorem} from \cite{Calegari_pickle}, is the following:

\begin{theorem}[Rationality Theorem, \cite{Calegari_pickle} p.15]\label{rationality_theorem}
Let $F$ be a free group.
\begin{enumerate}
\item{$\scl(g) \in \Q$ for all $g \in [F,F]$.}
\item{Every integral chain $\sum n_i g_i$ in $B_1(F)$ bounds an extremal surface}\label{extremal_bullet}
\item{The function $\scl$ is piecewise rational linear on $B_1(F)$}
\item{There is an algorithm to calculate $\scl$ on any finite dimensional rational
subspace of $B_1(F)$}
\end{enumerate}
\end{theorem}

In fact, in \cite{Calegari_pickle}, bullet (\ref{extremal_bullet}) merely says
that every $g \in [F,F]$ rationally bounds an extremal surface, but the argument
of the proof establishes the more general statement. The method of proof makes this
clear: let $X$ be a handlebody with $\pi_1(X) = F$, and let $\gamma_i$ be
loops in $X$ representing the free homotopy classes of the $g_i$. In \cite{Calegari_pickle}
it is shown that there is a simple branched surface $\mathcal{B}$, with boundary mapping to
$\cup_i \gamma_i$, which carries every admissible surface (after compression and
homotopy). The function $-\chi^-$ is a rational linear function of weights on
$\mathcal{B}$, and therefore $-\chi^-$ may be calculated on any rational class
by solving a linear programming problem. An extremal vector obtained e.g. by the
simplex method will be rational, and after scaling, is represented by an extremal surface.

We will also use the following technical Lemma, 
which is Lemma~4.2. from \cite{Calegari_pickle}:

\begin{lemma}\label{positive_wrap}
Let $S$ be a connected surface, and $f:S \to H$ an extremal surface rationally bounding
$\gamma$. Then there is another extremal surface $f':S' \to H$
rationally bounding $\gamma$, for which every component of $\partial S'$ maps to
$\gamma$ with positive degree.
\end{lemma}

The same argument shows that if $S$ bounds some collection $\sum \gamma_i$, one may replace $S$ if necessary
by another extremal surface for which every map of a boundary component of $S$
to every component $\gamma_i$ has positive degree. Such an extremal surface is
said to be {\em positive}. Hence in the sequel we will assume that all our extremal
surfaces are positive.

From our perspective, the importance of extremal surfaces is the following:

\begin{lemma}\label{extremal_injective}
Let $f:(S,\partial S) \to (X,\cup_i\gamma_i)$ be an extremal surface for $\sum n_i g_i$. 
Then $f,S$ is {\em incompressible} and {\em boundary incompressible}. 
That is, $f_*:\pi_1(S) \to \pi_1(X)$
is injective, and if $\alpha \subset S$ is an essential immersed proper arc with
endpoints on components $\partial_i,\partial_j$ of $\partial S$ both mapping to $\gamma_k$,
there is no arc $\beta \subset \gamma_k$ so that $f(\alpha) \cup \beta$ is homotopically
trivial in $X$.
\end{lemma}
\begin{proof}
Suppose $\alpha \subset S$ represents a conjugacy class in the kernel of $f_*$.
Since surface groups are LERF (\cite{Scott}), there is a finite cover $S'$ of
$S$ to which $\alpha$ lifts as an embedded loop. The lifted map $f':S' \to X$
is admissible, with $-\chi^-(S')/2n(S') = -\chi^-(S)/2n(S)$, so $f':S' \to X$
is also extremal. But $f'$ can be compressed along the (now embedded) loop
$\alpha$, reducing $-\chi^-$ while keeping $n(S')$ fixed, thereby contradicting
the fact that $f:S \to X$ was extremal.

Similarly, suppose $\alpha$ is an arc such that $f(\alpha) \cup \beta$
is homotopically trivial in $X$. Let $S'$ be a cover of $S$ in which
$\alpha$ is embedded. Let $S''$ be obtained from $S'$ by attaching a $1$-handle $R$
to $\partial \alpha$, and let $f'':S'' \to X$ be equal to $f'$ on $S'$,
and map the core of $R$ to $\beta$. Then $n(S'') = n(S')$. However, the union
of $\alpha$ with the core of $R$ is an essential embedded loop in $S''$ which
maps to a homotopically trivial loop in $X$. Hence we can compress this loop,
obtaining $f''':S''' \to X$ with $-\chi^-(S''')/2n(S''') < -\chi^-(S)/2n(S)$,
thereby contradicting the fact that $f:S \to X$ was extremal.
\end{proof}

A similar argument shows that if a closed surface realizes the Gromov--Thurston
norm in its homology class, it is injective.
In the sequel, by abuse of notation, we will use the phrase ``$S$ is injective'' to mean
that $f,S$ is incompressible and boundary incompressible.

\section{Surface subgroups}

\subsection{Graphs of free groups}

\begin{definition}
A {\em graph of groups} is a collection of groups indexed by the vertices
and edges of a connected graph, together with a family of injective homomorphisms from the
edge groups into the vertex groups. Formally, let $\Gamma$ be a connected graph. For
each vertex $v$ there is a vertex group $G_v$, and for each edge $e$ an
edge group $G_e$ so that for each inclusion $i:v \to e$ as an endpoint,
there is an injective homomorphism $\varphi_i:G_e \to G_v$.

The {\em fundamental group} $G$ of a graph $\Gamma$ of groups (as above) is defined as follows.
Let $G'$ be the group generated by all the groups $G_v$ and an element $e$ for each (oriented)
edge $e$ with relations that each edge element $e$ conjugates the subgroup
$i(G_e)$ of $G_v$ to the subgroup $j(G_e)$ of $G_w$, where $v$ is the
initial vertex of $e$ and $w$ is the final vertex, with respect to the
choice of orientation on $e$. Let $T$ be a maximal subtree of $\Gamma$. Then
define $G$ to be the quotient of $G'$ by the normal subgroup generated by elements
$e$ corresponding to edges of $T$.
\end{definition}

By abuse of notation, we sometimes say that $G$ {\em is a graph of groups}
with graph $\Gamma$. See e.g. Serre \cite{Serre_trees} \S~5.1. for more details.

In the sequel, let $G$ be a graph of groups with graph $\Gamma$ satisfying the
following properties:

\begin{enumerate}
\item{Every vertex group $G_v$ is free of finite rank}
\item{Every edge group $G_e$ is cyclic}
\item{The graph $\Gamma$ is finite}
\end{enumerate}

We say that such a group $G$ is {\em a graph of free groups amalgamated over
cyclic subgroups}. 

\subsection{Hyperbolic groups}

\begin{definition}
A path-metric space $X$ is {\em $\delta$-hyperbolic} for some $\delta \ge 0$
if for every geodesic triangle $abc$, the edge $ab$ is contained in the
(metric) $\delta$-neighborhood of the union of edges $ac \cup bc$.
\end{definition}

\begin{definition}
A group $G$ with a finite generating set $S$ is {\em word-hyperbolic}
(or just {\em hyperbolic} for short) if the Cayley graph $C_S(G)$ is
$\delta$-hyperbolic as a path metric space, for some finite $\delta$.
\end{definition}

Hyperbolic groups are introduced in \cite{Gromov_hyperbolic}, inspired in part
by work of Cannon, Epstein, Rips and Thurston. The theory of hyperbolic groups
is vast; the only property of hyperbolic groups we will need is that they do
not contain $\Z \oplus \Z$ or Baumslag--Solitar subgroups. Here the Baumslag--Solitar
group $B(p,q)$ ($p,q \ne 0$) is given by the presentation
$$B(p,q):=\langle a,b \; | \; ba^pb^{-1} = a^q\rangle$$
Note that $B(1,1) = \Z \oplus \Z$ as a special case.

\subsection{Construction of surface subgroups}

We are now in a position to state the main theorem of this paper.

\begin{theorem}
Let $G$ be a graph of free groups amalgamated over cyclic subgroups. If
$G$ is word-hyperbolic, and $H_2(G;\Q)$ is nonzero, then $G$ contains a
closed hyperbolic surface subgroup. Furthermore, the unit ball of the Gromov--Thurston
norm in $H_2(G;\R)$ is a finite-sided rational polyhedron.
\end{theorem}
\begin{proof}
We build a space $X$ with $\pi_1(X) = G$ as follows. For each vertex $v$
let $H_v$ be a handlebody with $\pi_1(H_v) = G_v$. For each edge $e$
let $A_e$ be an annulus. For each $i:v \to e$ let $\gamma_i \subset X$
be an embedded loop representing the conjugacy class of the generator
of $i(G_e)$, and glue the corresponding boundary component of $A_e$ to $H_v$
along $\gamma_i$. The Seifert van-Kampen theorem justifies the equality
$\pi_1(X) \cong G$. In fact, since each $H_v$ and $A_e$ is a $K(\pi,1)$,
and since the edge homomorphisms are all injective, the space $X$ itself is
a $K(\pi,1)$. See e.g. \cite{Hatcher}, Theorem~1B.11. p.92. Hence
$H_2(G;\Lambda) = H_2(X;\Lambda)$ for all coefficient groups $\Lambda$.

Let $E$ denote the union of the cores of the annuli $A_e$. Let $V = X - E$
and let $N$ be a regular neighborhood of $E$. The Mayer--Vietoris sequence
contains the following exact subsequence
$$H_2(V) \oplus H_2(N) \to H_2(X) \to H_1(V \cap N) \to H_1(V) \oplus H_1(N)$$
Since $H_2(V) = H_2(N) = 0$, it follows that an element of $H_2(X)$ is determined
by its image in $H_1(V \cap N)$. Geometrically, let $Y$ be obtained from $X$
by crushing each $H_v$ and a cocore of each $A_e$ to a point. Then $Y$
is a wedge of $S^2$'s, one for each $A_e$. The induced map $H_2(X) \to H_2(Y)$ is
an injection, and an element of $H_2(X)$ is determined by the degree with
which it maps over each sphere summand of $Y$.

Let $A$ be a nonzero class in $H_2(X)$ represented by a map of a closed surface
$f:S \to X$. If we make $f$ transverse to the core of each $A_e$ and adjust by
a homotopy, we can assume that $S_e:= f^{-1}(A_e)$ is a union of subsurfaces of $S$
each mapping properly to $A_e$. If $S_e^i$ is a component of $S_e$, the degree
of $f:S_e^i \to A_e$ is equal to the number of times $\partial S_e^i$ winds
(with multiplicity) around either boundary component of $A_e$. If some $S_e^i$
maps to some $A_e$ with degree $0$, compress a suitable subsurface of 
$S_e^i$ and push it off $A_e$ by a homotopy.

For each $v$, let $\gamma_1,\cdots,\gamma_m$ denote the
set of loops in $H_v$ which are the boundaries of components of the various $A_e$.
The surface $S_v: = f^{-1}(H_v)$ maps to $H_v$ with boundary wrapping various
times around the various $\gamma_i$. Let $n_i \in \Z$ be such that $S_v$
is an admissible surface bounding $\sum n_i \gamma_i$. Note that the $n_i$
(for various $v$)
are determined by the homology class $A$, and are precisely the coefficients
of the element $\partial A \in H_1(V \cap N)$ with respect to a basis for
$H_1(V \cap N)$ consisting of the various $\gamma_i$.

\vskip 12pt

For each $v$, let $g_v:T_v \to X$ be an extremal surface for $\sum n_i g_i$.
Note that $\partial T_v$ represents $n_v \sum n_i g_i$ in $H_1(V \cap N)$
for some integer $n_v$. If the various $T_v$ could be glued together along their
boundary components compatibly with $\Gamma$, the components of the
resulting surface would be injective, and their union would
map to $X$, representing a multiple of the class $A$ in $H_2(X)$. If $G$
is word-hyperbolic, we will show how to construct suitable covers of the
$T_v$ which can in fact be glued up.

\begin{lemma}\label{covering_lemma}
Let $S$ be an orientable surface with nonempty boundary components $\partial_i S$.
For $N \in \Z$, let $\phi:\partial_i \to \Z/N\Z$ be some function. If
$\sum_i \phi(\partial_i) = 0 \in \Z/N\Z$ then $\phi$ extends to a function
$\pi_1(S) \to \Z/N\Z$ whose kernel defines a regular cover $S'$ of $S$ with the
property that each boundary component $\partial_{ij}$ in the preimage of 
$\partial_i$ maps to $\partial_i$ with degree equal to the order of
$\phi(\partial_i)$ in $\Z/N\Z$.
\end{lemma}
\begin{proof}
Homomorphisms from $\pi_1(S)$ to abelian groups are exactly those which factor
through the abelianization $H_1(S)$. The components $\partial_i$ determine elements
of $H_1(S)$ which are subject only to the relation $\sum \partial_i = 0 \in H_1(S)$.
This follows directly from the exact sequence in relative homology
$$H_2(S) \to H_2(S,\partial S) \to H_1(\partial S) \to H_1(S)$$
together with the fact that $H_2(S) = 0$ and $H_2(S,\partial S) = \Z$.

The other statements are standard facts from the theory of covering spaces.
\end{proof}

By invoking Lemma~\ref{covering_lemma} repeatedly, we will construct covers of
the $T_v$ which can be glued up over the various $A_e$ one by one.
Let $e$ be an edge, and $v,w$ the end vertices. Let $\gamma \in H_v$ and
$\delta \in H_w$ be the loops along which the boundary components of
$A_e$ are attached. Suppose we have surfaces $T$ and $U$ mapping to $X$
and subsets $\partial_\gamma T, \partial_\delta U$ of the boundary
components which map to $\gamma$ and $\delta$ respectively. By
Lemma~\ref{positive_wrap} we can assume that each component of $\partial_\gamma T$
maps to $\gamma$ with positive degree, and similarly for $\partial_\delta U$.
Note that we should allow the possibility that $T=U$.

Assume for the moment that $\chi(T) < 0$ and $\chi(U) < 0$.
If $S$ is a surface of negative Euler characteristic, then $S$ admits a finite
cover with positive genus. Furthermore, if $S$ has positive genus, then $S$ admits
a degree $2$ cover $S'$ for which every boundary component of $S$ has exactly
two preimages in $S'$, each of which maps with degree $2$. So without loss of
generality, we can assume that the components of $\partial_\gamma T$ come in
pairs which each map to $\gamma$ with the same degree, and similarly for the
components of $\partial_\delta U$.

Let $N$ be the least common multiple of the degrees of maps from components
of either $\partial_\gamma T$ or $\partial_\delta U$ to $\gamma$ or $\delta$.
We will define homomorphisms $\phi:\pi_1(T) \to \Z/N\Z$ and $\psi:\pi_1(U) \to \Z/N\Z$
as follows. If $\nu,\nu'$ are a pair of components of $\partial_\gamma T$ mapping to $\gamma$
with the same degree $d$, then define $\phi(\nu) = d$ and $\phi(\nu') = -d$,
and define $\psi$ similarly on pairs of components of $\partial_\delta U$. 
Note that $\phi,\psi$ may be extended to have the value $0$ on components
of $\partial T$ and $\partial U$ not appearing in $\partial_\gamma T$ or
$\partial_\delta U$. Let $T',U'$ be the corresponding covers. Then by
construction, every component of $\partial_\gamma T'$ maps to $\gamma$ with
degree $N$, and every component of $\partial_\delta U'$ maps to $\delta$ with
degree $N$, so $U'$ and $T'$ can be glued up along $\partial_\gamma T'$ and
$\partial_\delta U'$ and their maps to $X$ extended over $A_e$.
Proceeding inductively, we can construct a surface $S'$ and a map $f':S' \to X$
representing some integral multiple of the class $A$. Since $S'$ is made by
gluing covers of injective maps, the definition of injective and
the Seifert van-Kampen theorem implies that every component of $S'$ is injective.

If $\chi(T)=0$ for some (possibly intermediate) surface $T$, then $T$ consists of
a union of annuli. If not every component $T$ is being glued up to $U$, one can
still take covers of these annuli as above and glue up. The only potentially
troublesome case is when $T=U$, and the free boundary components of $T$ are
being glued up to each other. But in this case, $G$ contains the mapping cylinder
of an injective map from $\Z$ to itself; i.e. it contains a $\Z \oplus \Z$ or
a Baumslag--Solitar group, and is therefore not hyperbolic. This proves the
first part of the theorem.

\vskip 12pt

To prove the statement about Gromov--Thurston norms, observe that if
$f'_*[S'] = n[S]$ in homology, then by construction $-\chi^-(S')/n \le -\chi^-(S)$.
Hence $S'$, as constructed, realizes the Gromov--Thurston norm in its homology
class. Note that this gives another proof that $S'$ is injective.
In fact, for each $H_v$, let $B_v \subset B_1(G_v)$ be the subspace spanned
by the $\gamma_i$ along which various $A_e$ are attached. The boundary map in
the Mayer--Vietoris sequence defines an integral linear injection
$H_2(X) \xrightarrow{\partial} \bigoplus_v B_v$ with components $\partial_v(A) \in B_v$,
and by the construction above, 
$$\|A\|_1 = 4 \cdot \sum_v \scl(\partial_v A)$$
Since each $\partial_v$ is an integral linear map, and the $\scl$ pseudo-norm on
each $B_v$ is a rational piecewise-linear function, the $L^1$ norm (i.e. the
Gromov--Thurston norm) on $H_2(X;\R)$ is a piecewise rational linear function.
Since $G$ is hyperbolic, the unit ball is a (nondegenerate) finite sided rational
polyhedron.
\end{proof}

\begin{remark}
If $G$ is not necessarily word-hyperbolic, it is nevertheless true 
(essentially by the argument above) that the Gromov--Thurston pseudo-norm on
$H_2(G;\R)$ is piecewise rational linear. Moreover, the same argument shows that
for any $G$ obtained as a graph of free groups amalgamated over cyclic subgroups,
and for any homology class $A \in H_2(G;\Q)$, either some multiple of $A$ is
represented by an injective closed hyperbolic surface, or $\|A\|_1=0$.
\end{remark}

\begin{remark}
If $M$ is a compact $3$-manifold, every integral class $A$ in $H_2(\pi_1(M);\Z)$
is represented by an {\em embedded} surface $S$ which realizes the infimum of
$-\chi^-$ in its projective class. It follows that the Gromov--Thurston norm 
(with Gromov's normalization) takes on values in $4\Z$ on $H_2(\pi_1(M);\Z)$.
This by itself ensures that the unit ball is a rational polyhedron. However, for
$G$ a graph of free groups as above, the Gromov--Thurston norm can take on
rational values with arbitrary denominators on elements of $H_2(G;\Z)$, so the
polyhedrality of the norm is more subtle. See \cite{Calegari_scl}, \S~4.1.9.
\end{remark}

\section{Acknowledgment}
While writing this paper I was partially supported by NSF grant DMS 0405491.
I would like to thank Mladen Bestvina, Benson Farb and Dongping Zhuang for
comments and some discussion. The work in this paper builds largely on work
done in \cite{Calegari_pickle} which benefited greatly from many discussions
with Jason Manning.

\end{document}